\documentclass[12pt]{article}

\usepackage{amsmath,amsthm,amssymb}
\usepackage[colorlinks=true,linkcolor=blue,citecolor=blue,pdfpagelabels=false,pdfstartview=FitH]{hyperref}
\usepackage{graphicx,color}
\usepackage{enumitem}

\def\R{\mathbb{R}}
\def\N{\mathbb{N}}
\def\tto{\rightrightarrows}

\newtheorem{definition}{Definition}[section]
\newtheorem{theorem}[definition]{Theorem}
\newtheorem{lemma}[definition]{Lemma}
\newtheorem{proposition}[definition]{Proposition}
\newtheorem{corollary}[definition]{Corollary}
\newtheorem{remark}[definition]{Remark}
\newtheorem{conjecture}[definition]{Conjecture}

\newenvironment{enum}
{\begin{enumerate}[topsep=1pt,parsep=0pt,itemsep=1pt]}{\end{enumerate}}

\begin{document}

\title{\bf Global convergence of a non-convex Douglas-Rachford iteration}
\author{\and Francisco J. Arag\'on Artacho\thanks{Centre for Computer Assisted Research Mathematics and its Applications (CARMA), University of Newcastle, Callaghan, NSW 2308, Australia.
\url{francisco.aragon@ua.es}}
\and Jonathan M. Borwein\thanks{Centre for Computer Assisted Research Mathematics and its Applications (CARMA), University of Newcastle, Callaghan, NSW 2308, Australia.
\url{jonathan.borwein@newcastle.edu.au}. Distinguished Professor King Abdul-Aziz Univ, Jeddah.}
}

\date{\today} \maketitle

\begin{abstract}
We establish a region  of convergence for the  proto-typical non-convex Douglas-Rachford iteration which finds a point on the intersection of a line and a circle. Previous work on the non-convex iteration \cite{BS11} was only able to establish local convergence, and was ineffective in that no explicit region of convergence could be given.


\end{abstract}

\section{Introduction}

The Douglas-Rachford algorithm is an iterative method for finding a point in the intersection of two (or more) closed sets. It is well-known that the iteration (weakly) converges
when it is applied to convex subsets of a Hilbert space (see e.g.~\cite[Fact~5.9]{BCL02} and the references therein). Despite the absence of a theoretical justification, the
algorithm has also been successfully applied to various non-convex practical problems, see e.g.~\cite{ERT07,GE08}.

An initial step towards providing some theoretical explanation of the
convergence in the non-convex case can be found in~\cite{BS11}, where the authors study a prototypical non-convex two-set scenario in which one of the sets is the Euclidean sphere
and the other is a line (or more generally, a proper affine subset). Similar to the convex case, Borwein and Sims prove local convergence of the algorithm to a point whose
projection into any of the sets gives a point in the intersection, whenever the sets intersect at more than one point and to a point outside the set in the tangential case (otherwise, the scheme diverges).
 Our aim herein is to extend their local result to a global one.

\begin{figure}[htb]
\begin{center}
\includegraphics[width=8cm]{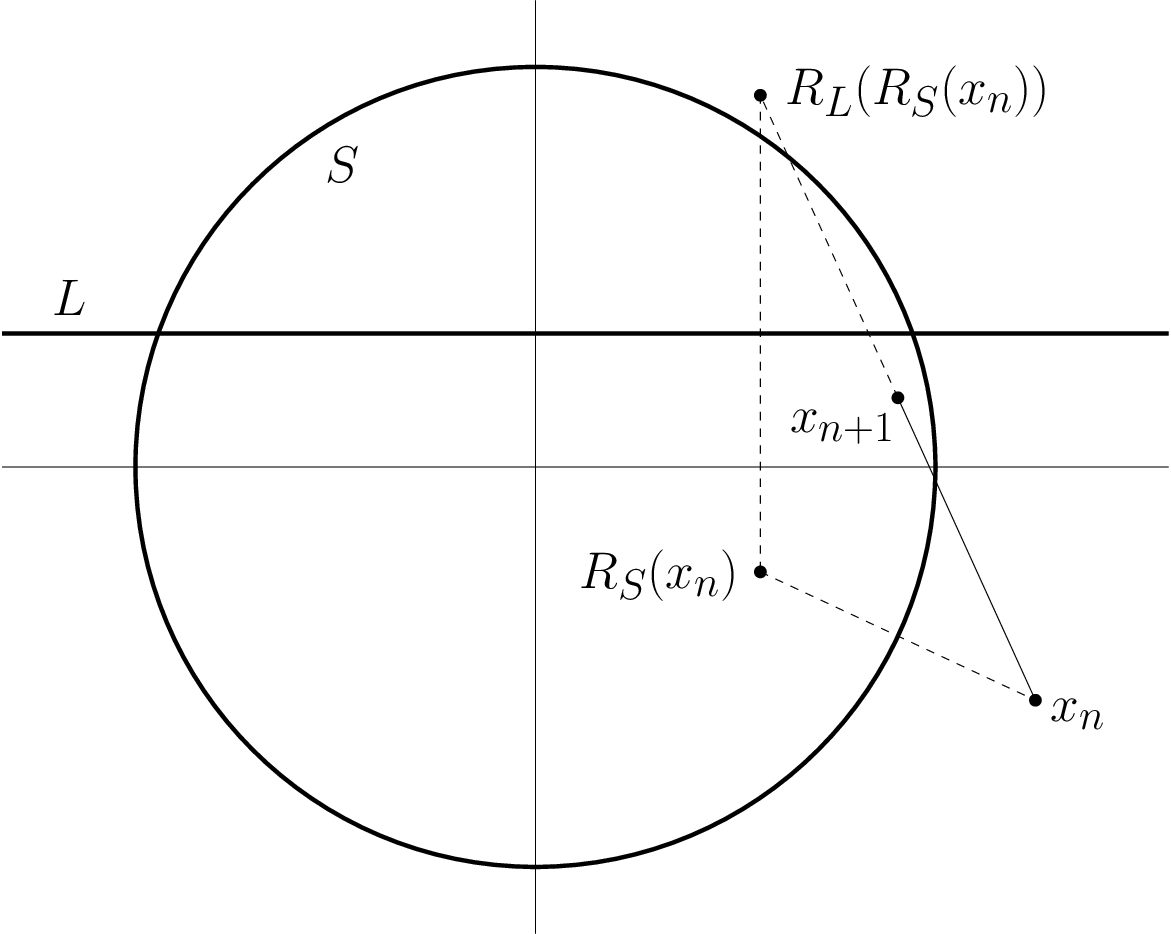}
\end{center}
\caption{Construction of the next iterate of the Douglas-Rachford scheme when applied to a sphere $S$ and a line $L$.}
\label{fig:construction}
\end{figure}

Given two closed subsets $A$ and $B$ of a Hilbert space $X$, the Douglas-Rachford scheme consists of first reflecting the current iteration in one of the two sets,  then reflecting the resulting
point in the other set, and then taking the average with the current iterate to form the next step (see Figure~\ref{fig:construction}). The \emph{reflection} of a point $x\in X$
in the set $A$ can be defined as
$$R_A(x):=2P_A(x)-x,$$
where $P_A(x)$ is the \emph{closest point projection} of the point $x$ in $A$, that is,
$$P_A(x):=\left\{z\in A\colon \|x-z\|=\inf_{a\in A}\|x-a\|\right\}.$$
In general, the projection $P_A: X\tto A$ is a set-valued mapping.

 If $A$ is convex, the projection is uniquely defined for every point in $X$,  thus yielding a single-valued mapping (see e.g.~\cite[Theorem~4.5.1]{BZ05}). The \emph{Douglas-Rachford iterative scheme} is defined as
\begin{equation}\label{DR_general}
x_{n+1}:=T_{A,B}(x_n),
\end{equation}
with $T_{A,B}:=\frac{1}{2}(R_B R_A+I)$, where $I$ is the identity map. Therefore, when the sets $A$ and $B$ are both convex, the iteration~\eqref{DR_general} is uniquely defined. Furthermore, if $A\cap B\neq\emptyset$, the sequence is weakly convergent to a fixed point $\bar x$ of the mapping $T_{A,B}$. Then, one can obtain a point in the intersection of $A$ and $B$ by projecting the point $\bar x$ in
the set $A$, since
\begin{align*}
T_{A,B}(\bar x)=\bar x&\iff \bar x=2P_B(2P_A(\bar x)-\bar x)-2P_A(\bar x)+\bar x\\
&\iff P_A(\bar x)=P_B(2P_A(\bar x)-\bar x),
\end{align*}
which implies $P_A(\bar x)\in A\cap B$.

Weak convergence of the algorithm comes from the fact that the projection mapping is \emph{firmly nonexpansive}, that is,
$$\|P_A(x)-P_A(y)\|^2+\|(I-P_A)(x)-(I-P_A)(y)\|^2\leq\|x-y\|^2,\quad\text{for all }x,y\in X,$$
see e.g.~\cite[Theorem~12.2]{GK90}, which implies that the reflection map is \emph{nonexpansive},
$$\|R_A(x)-R_A(y)\|\leq\|x-y\|,\quad\text{for all }x,y\in X,$$
whence, $T_{A,B}$ is firmly nonexpansive, see e.g.~\cite[Theorem~12.1]{GK90} or~\cite[Lemma~1]{LM79}, and this implies
the weak convergence of the iterative scheme~\eqref{DR_general}, see~\cite[Theorem~1]{O67}.

As in~\cite{BS11}, we restrict our study to the non-convex case of the intersection of a sphere $S:=\{x\in X\colon\|x\|=1\}$ and a line $L:=\{x=\lambda a+\alpha b\colon\lambda\in\R\}$ where, without loss of generality, one can take $\|a\|=\|b\|=1$, with $a$ orthogonal to $b$, and $\alpha>0$. If $X$ is $N$-dimensional, and $(x(1),x(2),\ldots,x(N))$ denotes the coordinates of $x$ relative to an orthonormal basis whose first two elements are respectively $a$ and $b$, the Douglas-Rachford iteration~\eqref{DR_general} becomes,
\begin{equation}\label{DR}\begin{array}{l}
x_{n+1}(1)=x_n(1)/\rho_n,\\
x_{n+1}(2)=\alpha+(1-1/\rho_n)x_n(2), \text{ and}\\
x_{n+1}(k)=(1-1/\rho_n)x_n(k),\text { for } k=3,\ldots,N,
\end{array}
\end{equation}
where $\rho_n:=\|x_n\|:=\sqrt{x_n(1)^2+\ldots+x_n(N)^2}$, see~\cite{BS11} for details. Borwein and Sims prove the next local convergence result.
\begin{theorem}[{\cite[Theorem~2]{BS11}}]\label{th:local}
  If $0\leq\alpha<1$ then the Douglas-Rachford scheme~\eqref{DR} is locally convergent at each of the points $\pm\sqrt{1-\alpha^2}a+\alpha b$.
\end{theorem}

\noindent Moreover, Borwein and Sims also conjecture the following:

\begin{conjecture}[{\cite[Conjecture 1]{BS11}}]\label{conj:global}
  In the simple example of a sphere and a line with two intersection points, the basin of attraction is the two open half-spaces forming the complement of the singular manifold $\langle x,a\rangle =0$.
\end{conjecture}

Our main objective is make progress on Conjecture~\ref{conj:global}. We shall follow an algebraic approach, and in order to avoid an even more involved analysis \begin{quote}\emph{we restrict our current study
to the  case where $N=2$ and $\alpha=1/\sqrt{2}$.}\end{quote}

This case is enough to expose all of the difficulties in establishing  Conjecture~\ref{conj:global}.
 Similar proofs can be obtained for all other cases, although we believe that a different non-algebraic approach is needed to provide simpler proofs for the general case.

\section{Convergence}

In order to ease the notation, we shall denote the coordinates of the current iteration with respect to the orthonormal basis $\{a,b\}$ by $(x_n,y_n)$. Then the iteration~\eqref{DR} for $N=2$ becomes
\begin{equation}\label{DR_2}\left\{\begin{array}{l}
\displaystyle x_{n+1}=\frac{x_n}{\rho_n}=\cos\theta_n=\rho_{n+1}\cos\theta_{n+1},\\
y_{n+1}=\alpha+\left(1-\frac{1}{\rho_n}\right)y_n
=\alpha+(\rho_n-1)\sin\theta_n=\rho_{n+1}\sin\theta_{n+1},
\end{array}\right.
\end{equation}
where $\rho_n=\sqrt{x_n^2+y_n^2}$ and $\theta_n$ is the argument of $(x_n,y_n)$. Throughout this section we  assume, as we indicated, that $\alpha=1/\sqrt{2}$, in which case the sphere $S$ and the line $L$ intersect at the points $(\alpha,\alpha)$ and $(-\alpha,\alpha)$. (Geometrically, this appears totally general. Algebraically, we have been unable to show this.)

On one hand, observe that whenever the iterations lie outside the sphere, the algorithm behaves as if the sphere was a ball, and therefore the mapping $T_{S,L}$ is nonexpansive.
On the other hand, when the iterations lie inside the ball, the mapping $T_{S,L}$ can be expansive. Despite this, we will show that the \emph{contractive behavior} of the sequence
in some regions overcomes the \emph{expansive behavior} occurring in other areas, and the sequence generated does converge. It is this oscillatory behaviour which demands better understanding.

 Because of  symmetry we will need analyze the case
when $x_0>0$, where we will show that the sequence converges to the point $(\alpha,\alpha)$. We thus study the behavior of the iterations in seven different regions, see
Figure~\ref{fig:regions}.

\begin{figure}[htb]
\begin{minipage}{.6\textwidth}
$$P_0:=\{(x,y)\in \R^2\mid y\leq 0<x\}$$
$$P_1:=\{(x,y)\in \R^2\mid x^2+y^2\leq 1\text{ and }0< y\leq x\}$$
$$P_2:=\{(x,y)\in \R^2\mid x^2+y^2> 1\text{ and }0< y\leq \alpha\}$$
$$P_3:=\{(x,y)\in \R^2\mid \alpha< y\leq x\}$$
$$P_4:=\{(x,y)\in \R^2\mid x^2+y^2> 1\text{ and }0< x< y\}$$
$$P_5:=\{(x,y)\in \R^2\mid x^2+y^2\leq 1,x> 0\text{ and }y>\alpha\}$$
$$P_6:=\{(x,y)\in \R^2\mid 0< x< y\leq \alpha\}$$\medskip
\end{minipage}
\begin{minipage}{.35\textwidth}
\begin{center}
\includegraphics[width=4.2cm]{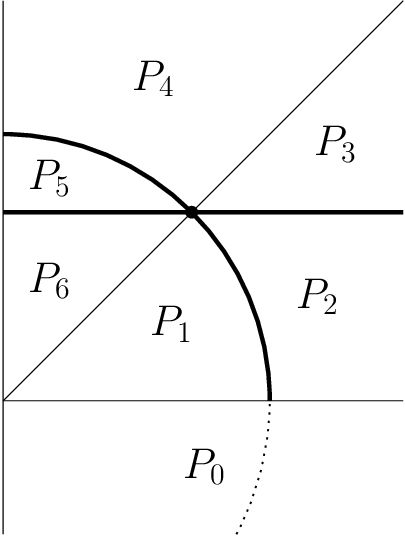}
\end{center}
\end{minipage}\caption{Different regions in the half-space $x>0$.}
\label{fig:regions}
\end{figure}

Our main result reads as follows.

\begin{theorem}\label{th:main}
  If $(x_0,y_0)\in\,\left[\varepsilon,1\right]\times[0,\sqrt{1-\epsilon^2}]$, with $\varepsilon:=\left(1-2^{-1/3}\right)^{3/2}\approx 0.0937$, then the sequence generated by the Douglas-Rachford scheme~\eqref{DR_2} with starting point
  $(x_0,y_0)$ is convergent to the point $(\alpha,\alpha)$.
\end{theorem}

In order to prove Theorem~\ref{th:main}, we will show that the sequence is convergent to $(\alpha,\alpha)$ whenever the initial point $(x_0,y_0)\in P_1$. The next step will be to
prove that the sequence hits the region $P_1$ after a finite number of iterations, when $(x_0,y_0)$ belongs to the other demarcated  areas. This will imply the convergence. In order to
demonstrate  convergence within the region $P_1$, we will analyze the behavior of the iterations within each of the other regions. We will show that the iterations pass through
the different regions in a counterclockwise way. We begin with the region $P_0$: the next proposition shows that the sequence must eventually abandon the region $P_0$.

\begin{proposition}\label{P0}
  If $(x_n,y_n)\in P_0$ then $y_{n+k}>0$ for some $k\in\N$.
\end{proposition}

\begin{proof}
  If $y_{n+m}=0$ for some $m=0,1,2,\ldots$, then $y_{n+m+1}=\alpha$, and we can take $k=m+1$. Suppose by contradiction that $y_{n+k}<0$ for all $k=0,1,2,\dots$. From
  $$0>y_{n+k+1}=\alpha+\left(1-\frac{1}{\rho_{n+k}}\right)y_{n+k},$$
  we deduce that $\rho_{n+k}>1$ for all $k=0,1,\ldots$. Thus
  $$x_{n+k+1}=\frac{x_{n+k}}{\rho_{n+k}}<x_{n+k},$$
  for all $k=0,1,\ldots$, and therefore the sequence $x_{n+k}$ is convergent to some $x^*$; whence, $\rho_{n+k}$ is convergent to $1$. Since
  $$y_{n+k+1}=\alpha+(\rho_{n+k}-1)\sin\theta_{n+k},$$
  one has that the sequence $y_{n+k}$ is convergent to $\alpha$, a contradiction with the assumption $y_{n+k}<0$ for all $k=0,1,2,\dots$.
\end{proof}

\begin{remark}\label{rem:y=0}
    If $y_0=0$ then $\rho_0=x_0$, and thus, $(x_1,y_1)=(1,\alpha)\in P_2$. It is possible to compute symbolically the sequence and check that $(x_6,y_6)\in P_1$.
\end{remark}

In the region $P_1\cup P_2\cup P_3$ the sequence gets closer after each iteration to the intersection point $(\alpha,\alpha)$, and after one or at most two iterations it jumps to
the next region.

\begin{lemma}\label{lemma1}
If $(x_n,y_n)\in P_1\cup P_2\cup P_3$ then
  \begin{equation}\label{bound_1}
     |(x_{n+1},y_{n+1})-(\alpha,\alpha)|^2\leq \frac{1}{2}|(x_n,y_n)-(\alpha,\alpha)|^2.
  \end{equation}
  Moreover, we have the following:
  \begin{enum}
    \item if $(x_n,y_n)\in P_1$ then $(x_{n+2},y_{n+2})\in P_2\cup P_3\cup P_4;$
    \item if $(x_n,y_n)\in P_2$ then $(x_{n+1},y_{n+1})\in P_3\cup P_4$;
    \item if $(x_n,y_n)\in P_3$ then $(x_{n+1},y_{n+1})\in P_4$.
  \end{enum}
\end{lemma}

\begin{proof}
To prove~\eqref{bound_1} notice that
\begin{align*}
2\left[(x_{n+1}-\alpha)^2+(y_{n+1}-\alpha)^2\right]&=2(\cos\theta_n-\alpha)^2+2(y_n-\sin\theta_n)^2\\
&=2(\cos\theta_n-\alpha)^2+2(\rho_n-1)^2\sin^2\theta_n\\
&=2\cos^2\theta_n-4\alpha\cos\theta_n+2\alpha^2+2\sin^2\theta_n\\
&\quad+2\rho_n(\rho_n-2)\sin^2\theta_n\\
&=3-4(\alpha\cos\theta_n+\rho_n\sin^2\theta_n)+2\rho_n^2\sin^2\theta_n,
\end{align*}
and also,
\begin{align*}
(x_n-\alpha)^2+(y_n-\alpha)^2&=(\rho_n\cos\theta_n-\alpha)^2+(\rho_n\sin\theta_n-\alpha)^2\\
&=\rho_n\cos^2\theta_n-2\alpha\rho_n\cos\theta_n+\alpha^2+\rho_n^2\sin^2\theta_n\\
&\quad-2\alpha\rho_n\sin\theta_n+\alpha^2\\
&=1+\rho_n^2-2\alpha\rho_n(\cos\theta_n+\sin\theta_n).
\end{align*}
Thus~\eqref{bound_1} holds if and only if
\begin{equation*}
2-4(\alpha\cos\theta_n+\rho_n\sin^2\theta_n)+2\rho_n^2\sin^2\theta_n-\rho_n^2+2\alpha\rho_n(\cos\theta_n+\sin\theta_n)\leq 0,
\end{equation*}
or, equivalently,
\begin{align}
f(\rho_n,\theta_n):=&{\left(2\sin^2\theta_n -1\right)} \rho_n^{2}-{\left(4\sin^2\theta_n -
{\left(\sin\theta_n+\cos\theta_n\right)}\sqrt{2}\right)}\rho_n\nonumber\\
&-2\sqrt{2}\cos\theta_n+2
\leq 0.\label{eq3}
\end{align}
If $\theta_n=\pi/4$, then~\eqref{eq3} holds with equality.

 We will now prove that $f(\rho_n,\theta_n)<0$  for all $\rho_n\in[0,1]$ and
$\theta_n\in[0,\pi/4)$. Observe that $f(0,\theta_n)=2-2\sqrt{2}\cos\theta_n<0$ for all $\theta_n\in[0,\pi/4)$. On the other hand, notice that
$f(\rho_n,\theta_n)$ is a quadratic function in terms of $\rho_n$, whose leading coefficient $2\sin^2\theta_n -1$ is nonzero. Its discriminant is
equal to
\begin{equation}\label{eq4}
16\sin^{4}\theta_n - 8\sqrt{2}\sin^{3}\theta_n+8{\left(\sqrt{2}\cos\theta_n - 2\right)}\sin^{2}\theta_n + 4 \sin\theta_n \cos\theta_n -
8\sqrt{2}\cos\theta_n + 10.
\end{equation}
If we make the substitution $w:=\sin\theta_n$, then~\eqref{eq4} becomes
$$16w^{4} - 8 \sqrt{2} w^{3} + 8 {\left(\sqrt{1-w^{2}} \sqrt{2} - 2\right)} w^{2} + 4 \sqrt{1-w^{2}} w - 8  \sqrt{1-w^{2}} \sqrt{2} + 10.$$
Hence, the discriminant is equal to zero if and only if
$$16 w^{4} - 8  \sqrt{2} w^{3} - 16  w^{2} + 10 = -4  \sqrt{1-w^{2}} {\left(2  \sqrt{2} w^{2} + w - 2  \sqrt{2}\right)}.$$
Taking squares of both sides and dividing by $4$, we get
$$64 w^{8} - 64 \sqrt{2} w^{7} - 64 w^{6} + 80 \sqrt{2} w^{5} + 52 w^{4} - 72 \sqrt{2} w^{3} + 12  w^{2} + 16 \sqrt{2} w - 7=0.$$
The roots of the latter equation are
\begin{gather*}
-\frac{1}{4}  {\left(\sqrt{-2  \sqrt{2} + 1} + \sqrt{2} + 1\right)} \sqrt{2},\; \frac{1}{4} {\left(\sqrt{-2 \sqrt{2} + 1} - \sqrt{2} - 1\right)}
\sqrt{2}, \\
-\frac{1}{4}  {\left(\sqrt{2  \sqrt{2} + 1} - \sqrt{2} + 1\right)} \sqrt{2}, \; \frac{1}{4} {\left(\sqrt{2 \sqrt{2} + 1} + \sqrt{2} - 1\right)}
\sqrt{2}, \;\frac{1}{\sqrt{2}}  .
\end{gather*}
The first two roots are complex numbers, the third one is negative, and the fourth one is greater than $1/\sqrt{2}$. Remembering that
$w=\sin\theta_n$, with $\theta_n\in[0,\pi/4)$, we have that $w\in[0,1/\sqrt{2})$. When $\theta_n=0$, the discriminant~\eqref{eq4} is equal to
$-8\sqrt{2} + 10<0$; whence,~\eqref{eq4} is always negative for all $\theta_n\in[0,\pi/4)$. Therefore, $f(\rho_n,\theta_n)\leq 0$ for all
$\rho_n\in[0,1]$ and $\theta_n\in[0,\pi/4]$, which completes the first part of the proof.
We turn to  the second part.

\emph{To prove~(i)}, notice first that $\theta_n\in(0,\pi/4]$, since $x_n\in P_1$. 
Then $x_{n+1}=\cos \theta_n\in[\alpha,1)$ and
$$y_{n+1}=\alpha+(\rho_n-1)\sin \theta_n\in[0,\alpha).$$
Furthermore,
\begin{align*}
x_{n+2}^2+y_{n+2}^2&=\frac{x_{n+1}^2}{\rho_{n+1}^2}+\alpha^2+\left(1-\frac{1}{\rho_{n+1}}\right)^2y_{n+1}^2+2\alpha\left(1-\frac{1}{\rho_{n+1}}\right)y_{n+1}\\
&=1+\frac{1}{2}+\left(1-\frac{2}{\rho_{n+1}}\right)y_{n+1}^2+2\alpha\left(1-\frac{1}{\rho_{n+1}}\right)y_{n+1}\\
&=1+\left(y_{n+1}+\frac{1}{\sqrt{2}}\right)\left(y_{n+1}\left(\frac{\rho_{n+1}-2}{\rho_{n+1}}\right)+\frac{1}{\sqrt{2}}\right).
\end{align*}
Thus, $x_{n+2}^2+y_{n+2}^2>1$ if and only if
\begin{equation}\label{eq1}
y_{n+1}\left(\frac{2-\rho_{n+1}}{\rho_{n+1}}\right)<\frac{1}{\sqrt{2}}.
\end{equation}
On the other hand, $\rho_{n+1}\geq\sqrt{1/2+y_{n+1}^2}$, which implies that
$$y_{n+1}\left(\frac{2-\rho_{n+1}}{\rho_{n+1}}\right)\leq y_{n+1}\left(\frac{2-\sqrt{1/2+y_{n+1}^2}}{\sqrt{1/2+y_{n+1}^2}}\right).$$
Hence~\eqref{eq1} holds if
\begin{equation}\label{eq2}
y_{n+1}\left(\frac{2-\sqrt{1/2+y_{n+1}^2}}{\sqrt{1/2+y_{n+1}^2}}\right)<\frac{1}{\sqrt{2}},
\end{equation}
or, equivalently,
$$-{\sqrt{2y_{n+1}^{2} + 1} \sqrt{2} y_{n+1} + 4  y_{n+1}}<{\sqrt{2  y_{n+1}^{2} + 1}}.$$
Taking $w:=\sqrt{2y_{n+1}^2+1}\in[1,\sqrt{2})$, the above inequality becomes
$$\left(- w + 2\sqrt{2}\right)\sqrt{w^{2} - 1}  < w.$$
Squaring both sides we obtain
$$w^{4} - 4 \sqrt{2} w^{3} + 7  w^{2} + 4  \sqrt{2} w - 8 < w^{2},$$
that is,
$$w^{4} - 4 \sqrt{2} w^{3} + 6  w^{2} + 4  \sqrt{2} w - 8 < 0.$$
Finally, notice that
$$w^{4} - 4 \sqrt{2} w^{3} + 6  w^{2} + 4  \sqrt{2} w - 8 ={\left(w - \sqrt{2}\right)}^{2} {\left(w - \sqrt{2} - \sqrt{6}\right)} {\left(w - \sqrt{2} + \sqrt{6}\right)}<0,$$
since $w\in[1,\sqrt{2})$. Therefore~\eqref{eq1} holds, whence $(x_{n+1},y_{n+1})\in P_2\cup P_3\cup P_4$, and the proof of~(i) is complete.

\emph{To prove (ii)}, observe that
$$y_{n+1}=\alpha+\left(1-\frac{1}{\rho_n}\right)y_n>\alpha,$$
since $\rho_n>1$ and $y_n>0$. Moreover, by a similar argumentation as above, we have that
\begin{align*}
x_{n+1}^2+y_{n+1}^2&=
1+\left(y_n+\frac{1}{\sqrt{2}}\right)\left(y_n\left(\frac{\rho_n-2}{\rho_n}\right)+\frac{1}{\sqrt{2}}\right).
\end{align*}
Hence, $x_{n+1}^2+y_{n+1}^2>1$ if and only if
\begin{equation}\label{eq5}
\frac{1}{\sqrt{2}}>y_n\left(\frac{2-\rho_n}{\rho_n}\right)=(2-\rho_n)\sin\theta_n,
\end{equation}
which holds since $\sin\theta_n\in(0,\alpha)$ and $\rho_n>1$. Therefore $(x_{n+1},y_{n+1})\in P_3\cup P_4$, as claimed.

\emph{To prove (iii)}, if $(x_{n},y_n)\in P_3$, we have again $x_{n+1}^2+y_{n+1}^2>1$ because of~\eqref{eq5}, since $\sin\theta_n\in(0,\alpha]$ and $\rho_n>1$.
Furthermore,
$$x_{n+1}-y_{n+1}=\cos\theta_n-\alpha+(1-\rho_n)\sin\theta_n.$$
Since $\rho_n>\frac{\alpha}{\sin\theta_n}$ and $\theta_n\in(0,\pi/4]$, we have
$$x_{n+1}-y_{n+1}=(\cos\theta_n+\sin\theta_n)-\alpha-\rho_n\sin\theta_n< \sqrt{2}-2\alpha=0.$$
Thus, $(x_{n+1},y_{n+1})\in P_4$, and the proof is complete.
\end{proof}

In the region $P_4$ the sequence does not get any farther from the intersection point $(\alpha,\alpha)$, and one can show that it jumps to the next region $P_5$ after a finite
number of steps.

\begin{lemma}\label{P4} If $(x_n,y_n)\in P_4$ then
  \begin{equation}\label{bound_3}
     |(x_{n+1},y_{n+1})-(\alpha,\alpha)|^2 \leq |(x_n,y_n)-(\alpha,\alpha)|^2.
  \end{equation}
Moreover, $(x_{n+1},y_{n+1})\in P_4\cup P_5$ and there is some $k\in\mathbb{N}$ such that $(x_{n+k},y_{n+k})\in P_5$.
\end{lemma}

\begin{proof}
The first part is a direct consequence of the convex case: when the point lies outside the sphere, the Douglas-Rachford iteration behaves as it the sphere was a ball; thus, we
have the nonexpansive bound~\eqref{bound_3}, see e.g.~\cite[Theorem~12.1]{GK90} or~\cite[Lemma~1]{LM79}.

For the second part, observe that
$$y_{n+1}-\alpha=(\rho_n-1)\sin\theta_n>0,$$
since $\theta_n\in(\pi/4,\pi/2)$ and $\rho_n>1$. Furthermore,
$$x_{n+1}-y_{n+1}=(\cos\theta_n+\sin\theta_n)-\alpha-\rho_n\sin\theta_n<\sqrt{2}-2/\alpha=0;$$
thus, $(x_{n+1},y_{n+1})\in P_4\cup P_5$. Finally, if $(x_k,y_k)\in P_4$ for all $k> n$, then $\rho_k>1$  and $\theta_k\in(\pi/4,\pi/2)$, and hence
$$x_{k+2}=\frac{x_{k+1}}{\rho_{k+1}}<x_{k+1}=\frac{x_{k}}{\rho_k}< x_{k}<\ldots <x_n,$$
for all $k> n$. Similarly,
$$y_{k+2}=y_{k+1}+\alpha-\sin\theta_{k+1}<y_{k+1}=y_{k}+\alpha-\sin\theta_{k}<y_{k}<\ldots<y_n.$$
Therefore, the sequence $(x_k,y_k)$ is convergent, whence $\rho_k\to 1$, and then $(x_k,y_k)\to(\alpha,\alpha)$. Nevertheless,
$$x_{k+1}=\cos\theta_k<\alpha,$$
for all $k\geq n$; thus, the decreasing sequence $x_k$ cannot converge to $\alpha$, and we obtain a contradiction.
\end{proof}

In the regions $P_5$ and $P_6$ the sequence gets farther from the point $(\alpha,\alpha)$. Nevertheless, the effect is ``rather small'' compared to the contractiveness of other
regions.

\begin{lemma}\label{P56}
If $(x_n,y_n)\in P_5\cup P_6$ then
  \begin{equation}\label{bound_2}
     |(x_{n+1},y_{n+1})-(\alpha,\alpha)|^2< \left(\frac{5}{2}- \sqrt{2} + \frac{1}{2}  \sqrt{-20  \sqrt{2} + 29}\right)|(x_n,y_n)-(\alpha,\alpha)|^2.
  \end{equation}
  Moreover, we have the following:
  \begin{enum}
    \item if $(x_n,y_n)\in P_5$ then $(x_{n+1},y_{n+1})\in P_6$;
    \item if $(x_n,y_n)\in P_6$ with $x_n\geq \varepsilon:=\left(1-2^{-1/3}\right)^{3/2}\approx 0.0937$ then there is some $k\in\N$ such that $(x_{n+k},y_{n+k})\in P_1$.
  \end{enum}
\end{lemma}

\begin{proof}
Let $\eta:=\left(\frac{5}{2}- \sqrt{2} + \frac{1}{2} \sqrt{-20 \sqrt{2} + 29}\right)^{-1}= \frac{5}{2}- \sqrt{2} -\frac{1}{2} \sqrt{-20 \sqrt{2} +
29} $. By a similar computation to the one in Lemma~\ref{lemma1}, one can prove that~\eqref{bound_2} holds if and only if
\begin{align}\label{function_f}
f(\rho_n,\theta_n):=&{\left(\eta\sin^2\theta_n-1\right)}\rho_n^{2}-{\left(2\eta\sin^2\theta_n-\sqrt{2}(\sin\theta_n+\cos\theta_n)\right)}\rho_n\nonumber\\
&-{\left(\sqrt{2} \cos\theta_n - \frac{3}{2}\right)} \eta  - 1< 0,
\end{align}
for all $\rho_n\in(0,1]$ and $\theta_n\in(\pi/4,\pi/2)$. Again, $f(\rho_n,\theta_n)$ is a quadratic function in terms of $\rho_n$, whose leading
coefficient $\eta\sin^2\theta_n-1$ is negative, since $\eta<1$. After the change of variable $w:=\sin\theta_n$, its discriminant is equal to
\begin{align*}
2  {\left(2  w^{4} - 3  w^{2}\right)} \eta^{2} + 4  \sqrt{1-w^2}  {\left(\sqrt{2} \eta^{2} w^{2} - \sqrt{2}{\left( w^{2} + 1\right)} \eta + w\right)}\\
 - 2  {\left(2  \sqrt{2} w^{3} - 2  w^{2} - 3\right)} \eta - 2
\end{align*}
Thus, the discriminant is equal to zero if and only if
\begin{align*}
2 & {\left(2  w^{4} - 3  w^{2}\right)} \eta^{2}  - 2  {\left(2  \sqrt{2} w^{3} - 2  w^{2} - 3\right)} \eta - 2\\
&=- 4  \sqrt{1-w^2}  {\left(\sqrt{2} \eta^{2} w^{2}
- \sqrt{2}{\left( w^{2} + 1\right)} \eta + w\right)}.
\end{align*}
Taking squares and collecting the terms in $w$ we obtain
\begin{align*}
0=\; & 16  \eta^{4} w^{8} - 32  \sqrt{2}\eta^{3} w^{7} - 16  {\left(\eta^{4} + 2  \eta^{3} - 4  \eta^{2}\right)} w^{6} + 16  {\left(3
\sqrt{2}
\eta^{3} - 2  \sqrt{2} \eta\right)} w^{5}\\
&+ 4{\left(\eta^{4} + 8  \eta^{2} + 4\right)} w^{4}- 16  {\left(5  \sqrt{2} \eta^{2} - \sqrt{2} \eta\right)} w^{3} - 8  {\left(\eta^{3} - 5
\eta^{2} + 2  \eta + 2\right)} w^{2} \\
&+ 32 \sqrt{2} \eta w + 4\eta^{2} - 24  \eta + 4.
\end{align*}
By substituting $u:=\sqrt{2}w$, we get
\begin{align*}
0=\;&\eta^{4} u^{8} - 4  \eta^{3} u^{7} - 2  {\left(\eta^{4} + 2  \eta^{3} - 4  \eta^{2}\right)} u^{6} + 4  {\left(3  \eta^{3} - 2
\eta\right)} u^{5} + {\left(\eta^{4} + 8  \eta^{2} + 4\right)} u^{4} \\
&- 8  {\left(5  \eta^{2} - \eta\right)} u^{3}  - 4  {\left(\eta^{3}- 5  \eta^{2} + 2  \eta + 2\right)} u^{2} + 32  \eta u+ 4 \eta^{2} -
24  \eta + 4
\end{align*}
This polynomial has four real roots, of which three are positive: $1$ (double) and $\sqrt{2}$. Therefore, the only roots of the discriminant are $\alpha$ and $1$. Since
$w=\sin\theta_n\in(\alpha,1)$, we have that $f(\rho_n,\theta_n)$ has always the same sign for all $\rho_n\in(0,1]$ and $\theta_n\in(\pi/4,\pi/2)$. By checking for instance that
$f\left(1,\frac{3}{8}\pi\right)<0$ we complete the first part of the proof.

\emph{To prove (i)} observe that
$$y_{n+1}=\alpha+\left(1-\frac{1}{\rho_n}\right)y_n\leq\alpha.$$
Moreover,
$$x_{n+1}-y_{n+1}=\cos\theta_n-\alpha+(1-\rho_n)\sin\theta_n.$$
Since $\rho_n\in\left(\frac{\alpha}{\sin\theta_n},1\right]$, then
$$x_{n+1}-y_{n+1}< \cos\theta_n+\sin\theta_n-2\alpha\leq \sqrt{2}-2\alpha=0,$$
and we are done, because $x_{n+1}=\cos\theta_n\in(0,\alpha)$.%

\emph{Finally, to prove~(ii)}, we  begin by showing that if $(x_n,y_n)\in P_6$ with $x_n\geq\varepsilon$, then $(x_{n+1},y_{n+1})\in P_6\cup P_1$ with
$x_{n+1}>x_n\geq\varepsilon$, for $\varepsilon:=\left(1-2^{-1/3}\right)^{3/2}$. Indeed, observe that
$$x_{n+1}=\frac{x_n}{\rho_n}=\cos\theta_n\in\, ]x_{n},\alpha[,$$
since $\theta_n\in\,]\pi/4,\pi/2[$ and $\rho_n<1$, which also implies,
$$y_{n+1}=y_n+\alpha-\sin\theta_n<y_n.$$
Furthermore, one has
$$y_{n+1}=\alpha+\rho_n\cos\theta_n\frac{\sin\theta_n}{\cos\theta_n}-\sin\theta_n\geq\alpha+\varepsilon\tan\theta_n-\sin\theta_n\geq 0,$$
since the function $g(\theta):=\alpha+\varepsilon\tan\theta-\sin\theta$ is nonnegative when $\theta\in\, ]\pi/4,\pi/2[$ (it attains its minimum $0$ at $\theta=\arcsin(2^{-1/6})$).
If $y_{n+1}=0$, then by Remark~\ref{rem:y=0} we have that $(x_{n+7},y_{n+7})\in P_1$. Otherwise, $(x_{n+1},y_{n+1})\in P_6\cup P_1$.

Now assume by way of contradiction that
$(x_{n+k},y_{n+k})\in P_6$ for all $k\in\N$. The sequences $x_n$ and $y_n$, being bounded and strictly increasing/decreasing (respectively), must converge to some points $x^*$ and
$y^*$. By the definition of the Douglas-Rachford iteration~\eqref{DR_2} this implies that $\rho_n$ converges to $1$, and then $y^*=\alpha$, which in turns implies $x^*=\alpha$; that
is, the sequence $(x_{n+k},y_{n+k})$ converges to $(\alpha,\alpha)$. Therefore
$$\alpha>y_{n+k}>y_{n+k+1}>\ldots\geq y^*=\alpha,$$
a contradiction.
\end{proof}

\begin{remark}\label{chaos}
The condition $x_n\geq \left(1-2^{-1/3}\right)^{3/2}$ in (ii) is required in order to avoid the sequence ever again hitting the region $P_0$, where a more complex (and somehow chaotic)
behavior occurs.
\end{remark}

The next result shows that the iterations cannot stay above the 45 degrees line for an infinite number of steps.

\begin{proposition}\label{prop2}
  If $(x_n,y_n)$ is a sequence generated by the Douglas-Rachford algorithm, then for all $m\in\N$ there exists some $k\geq m$ such that $x_k\geq y_k$; that is, $(x_k,y_k)\in P_0\cup P_1\cup P_2\cup P_3$.
\end{proposition}

\begin{proof}
Assume by contradiction that $x_n<y_n$ for all $n\geq k$. Then $\theta_n\in(\pi/4,\pi/2)$ for all $n\geq k$, and thus,
\begin{equation}\label{eq0}
y_{n+1}=y_n+\alpha-\sin\theta_n<y_n.
\end{equation}
Hence, the sequence $y_n$ is strictly decreasing for all $n\geq k$, and being bounded from below by $0$, it must converge to some point $y^*$. Then, from~\eqref{eq0}, one must
have that $\sin\theta_n$ converges to $\alpha$.  From this we deduce that $\theta_n$ converges to $\pi/4$. This implies that $x_n$ converges to $\alpha$, and then $\rho_n$
converges to $1$. Thus, $y^*=\alpha$, and therefore, $(x_n,y_n)$ converges to $(\alpha,\alpha)$. Remembering that the sequence $y_n$ is strictly decreasing, from Lemma~\ref{P4}
and Lemma~\ref{P56}, we conclude that there is some $m\in\N$ such that $(x_n,y_n)\in P_6$ for all $n\geq m$. Being that $y_n$ strictly decreasing, it cannot converge to $\alpha$, and we
obtain yet again a contradiction.
\end{proof}

Now we have all the necessary ingredients to prove that if the initial point lies in $P_1$, the iteration again hits  $P_1$ after a finite number of steps, and moreover, the distance to the solution $(\alpha,\alpha)$ has by then  decreased.

\begin{theorem}\label{globalP1}
  If $(x_0,y_0)\in P_1$, then there is some $m\in\N$ such that $(x_m,y_m)\in P_1$, with
  \begin{equation}\label{eq7}
    |(x_k,y_k)-(\alpha,\alpha)|^2\leq \frac{\gamma^3}{4}|(x_0,y_0)-(\alpha,\alpha)|^2\leq 0.86 |(x_0,y_0)-(\alpha,\alpha)|^2.
  \end{equation}
  for all $k\leq m$, where $\gamma:=\frac{5}{2}- \sqrt{2} + \frac{1}{2}  \sqrt{-20  \sqrt{2} + 29}$.
\end{theorem}

\begin{proof}
Since $\theta_0\in(0,\pi/4]$, one has $x_1=\cos\theta_0\in[\alpha,1)$, and also
$$y_1=\alpha+(\rho_0-1)\sin\theta_0\in[0,\alpha).$$
Thus, $(x_1,y_1)\in P_1\cup P_2$. By Lemma~\ref{lemma1} we have $(x_2,y_2)\in P_2\cup P_3\cup P_4$, and moreover,
$$|(x_2,y_2)-(\alpha,\alpha)|^2\leq \frac{1}{2}|(x_1,y_1)-(\alpha,\alpha)|^2\leq \frac{1}{4}|(x_0,y_0)-(\alpha,\alpha)|^2.$$
Furthermore, by (ii) and (iii) in Lemma~\ref{lemma1} and also by Lemma~\ref{P4}, there is some $k\in\N$ such that $(x_{k-1},y_{k-1})\in P_5$, having
$$|(x_{k-1},y_{k-1})-(\alpha,\alpha)|^2\leq \frac{1}{4}|(x_0,y_0)-(\alpha,\alpha)|^2.$$
Finally, by Lemma~\ref{P56}, we get
\begin{equation}\label{eq6}
  |(x_k,y_k)-(\alpha,\alpha)|^2\leq \frac{\gamma}{4}|(x_0,y_0)-(\alpha,\alpha)|^2,
\end{equation}
and $(x_k,y_k)\in P_6$. Then
$$y_{k+1}-\alpha=(\rho_k-1)\sin\theta_k<0\quad\text{and}\quad x_{k+1}=\cos\theta_k\in(0,\alpha),$$
since $\theta_k\in(\pi/4,\pi/2)$. Thus, $y_{k+1}<\alpha$. On the other hand, $y_{k+1}>0$ if and only if
$$\rho_k>1-\frac{\alpha}{\sin\theta_k},$$
which holds if $\rho_k>1-\alpha$. Because of~\eqref{eq6}, and taking into account that $|(x_0,y_0)-(\alpha,\alpha)|\leq 1$, one has
\begin{equation}\label{eq11}
  (x_k-\alpha)^2+(y_k-\alpha)^2\leq \frac{\gamma}{4},
\end{equation}
and hence, 
$$\rho_k\geq\sqrt{2}\left(\alpha-\sqrt{\frac{\gamma}{8}}\right)>1-\alpha.$$
Therefore, $x_{k+1}\in(0,\alpha)$ and $y_{k+1}\in(0,\alpha)$; whence, $(x_{k+1},y_{k+1})\in P_6\cup P_1$. If $(x_{k+1},y_{k+1})\in P_1$, then by Lemma~\ref{P56} we have
$$|(x_{k+1},y_{k+1})-(\alpha,\alpha)|^2\leq \gamma |(x_k,y_k)-(\alpha,\alpha)|^2\leq \frac{\gamma^2}{4}|(x_0,y_0)-(\alpha,\alpha)|^2,$$
which implies~\eqref{eq7}, and we are done. Otherwise, $(x_{k+1},y_{k+1})\in P_6$. 
Since
$$x_{k+1}=\frac{x_k}{\sqrt{x_k^2+y_k^2}},$$
one has
\begin{equation}\label{reverse_x}
  x_k=\frac{x_{k+1}}{\sqrt{1-x_{k+1}^2}}y_k,
\end{equation}
because $x_k,y_k>0$. Then,
$$y_{k+1}=\alpha+\left(1-\frac{1}{\sqrt{x_k^2+y_k^2}}\right)y_k=\alpha+y_k-\sqrt{1-x_{k+1}^2},$$
and thus,
\begin{equation}\label{reverse_y}
  y_k=y_{k+1}-\alpha+\sqrt{1-x_{k+1}^2}.
\end{equation}
Now $(x_k,y_k)\in P_6$ implies $y_k\leq\alpha$; whence,
$$y_{k+1}\leq\sqrt{2}-\sqrt{1-x_{k+1}^2}.$$
Since $(x_{k+1},y_{k+1})\in P_6$, then $x_{k+1},y_{k+1}>0$, and we have~\eqref{reverse_x} and~\eqref{reverse_y} for $k+1$ instead of $k$. Thus, the latter inequality yields
$$y_{k+1}\leq\sqrt{2}-\sqrt{1-\frac{x_{k+2}^2}{1-x_{k+2}^2}y_{k+1}^2},$$
which implies
$$\left(\sqrt{2}-y_{k+1}\right)^2\geq1-\frac{x_{k+2}^2}{1-x_{k+2}^2}y_{k+1}^2,$$
or, equivalently,
\begin{align*}
\left(\sqrt{2} x_{k+2}^{2} + y_{k+1} - \sqrt{2  x_{k+2}^{4} - 3  x_{k+2}^{2} + 1} - \sqrt{2}\right)&\Big(\sqrt{2} x_{k+2}^{2}+ y_{k+1} \\
&+ \sqrt{2  x_{k+2}^{4} - 3  x_{k+2}^{2} + 1} - \sqrt{2}\Big)\geq 0
\end{align*}
Hence, one must have one of the following two possibilities:
\begin{gather}
y_{k+1}\geq\sqrt{2}-\sqrt{2} x_{k+2}^{2} + \sqrt{2  x_{k+2}^{4} - 3  x_{k+2}^{2} + 1}, \label{eq8}\\
y_{k+1}\leq\sqrt{2}-\sqrt{2} x_{k+2}^{2} - \sqrt{2  x_{k+2}^{4} - 3  x_{k+2}^{2} + 1}.\label{eq9}
\end{gather}

Let us show\emph{} that~\eqref{eq8} is not possible. Indeed, $x_{k+2}=\cos\theta_{k+1}\in(0,\alpha)$ since $(x_{k+1},y_{k+1})\in P_6$. Thus,
$$\sqrt{2}-\sqrt{2} x_{k+2}^{2} + \sqrt{2  x_{k+2}^{4} - 3  x_{k+2}^{2} + 1}\geq\frac{1}{\sqrt{2}}+\sqrt{2  x_{k+2}^{4} - 3  x_{k+2}^{2} + 1}\geq\alpha,$$
and hence,~\eqref{eq8} would imply $y_{k+1}\geq\alpha$, a contradiction. Therefore,~\eqref{eq9} holds, as claimed, and using~\eqref{reverse_y} for $k+1$ we get
$$y_{k+2}-\alpha+\sqrt{1-x_{k+2}^2}\leq\sqrt{2}-\sqrt{2} x_{k+2}^{2} - \sqrt{2  x_{k+2}^{4} - 3  x_{k+2}^{2} + 1},$$
that is,
\begin{equation}\label{eq10}
  y_{k+2}\leq\frac{3}{2}\sqrt{2}-\sqrt{2} x_{k+2}^{2}-\sqrt{1-x_{k+2}^2} - \sqrt{2  x_{k+2}^{4} - 3  x_{k+2}^{2} + 1}.
\end{equation}
Consider the function
$f(x):=\frac{3}{2}\sqrt{2}-\sqrt{2} x^{2}-\sqrt{1-x^2} - \sqrt{2  x^{4} - 3  x^{2} + 1}$.

 We now show that $f(x)$ has only one fixed point in $(0,\alpha)$.
First, we have $f(x)=x$ if and only if
$$-\frac{3}{2}\sqrt{2}+\sqrt{2} x^{2}+\sqrt{1-x^2} + \sqrt{\left(x^{2} - 1\right)\left(2 \, x^{2} - 1\right)}+x=0,$$
which becomes, after the change of variable $x=z/\sqrt{2}$,
$$-\frac{3}{2}  \sqrt{2}+\frac{1}{2}  \sqrt{2} z^{2} + \sqrt{-\frac{1}{2}  z^{2} + \frac{1}{2}} \sqrt{-\frac{1}{2}  z^{2} + 1} \sqrt{2} + \frac{1}{2}  \sqrt{2} z + \sqrt{-\frac{1}{2}  z^{2} + 1} =0,$$
and multiplying by $\sqrt{2}$,
$$z^{2} +z-3+ \sqrt{-z^{2} + 2} {\left(\sqrt{-z^{2} + 1} + 1\right)} =0,$$
that is,
$$z^{2} +z-3=-\sqrt{-z^{2} + 2} {\left(\sqrt{-z^{2} + 1} + 1\right)}.$$
Taking squares we get
$$z^{4} + 2  z^{3} - 5  z^{2} - 6  z + 9 = z^{4} - 2  \sqrt{-z^2+1} {\left(z^{2} - 2\right)} - 4  z^{2} + 4,$$
and then,
$$2  z^{3} - z^{2} - 6  z + 5 = -2  \sqrt{-z^{2} + 1} {\left(z^{2} - 2\right)}.$$
Taking squares again we obtain
$$4  z^{6} - 4  z^{5} - 23  z^{4} + 32  z^{3} + 26  z^{2} - 60  z + 25 = -4  z^{6} + 20  z^{4} - 32  z^{2} + 16,$$
and grouping terms,
$$8  z^{6} - 4  z^{5} - 43  z^{4} + 32  z^{3} + 58  z^{2} - 60  z + 9=0,$$
which is equivalent to
$${\left(z - 1\right)} {\left(8  z^{5} + 4  z^{4} - 39  z^{3} - 7  z^{2} + 51  z - 9\right)}=0.$$
By Sturm's theorem, the polynomial $8  z^{5} + 4  z^{4} - 39  z^{3} - 7  z^{2} + 51  z - 9$ has only one root $\hat z$ in the interval $(0,1)$, which can be numerically computed:
$$\hat z\approx 0.186012649543.$$
Therefore, $f(x)$ has a unique fixed point $\hat x$ in $(0,\alpha)$, with
$$\hat x\approx 0.131530805878.$$
In particular, $f(x)<x$ for all $x\in(0.14,\alpha)$. Thus, by~\eqref{eq10}, if $x_{k+2}\in (0.14,\alpha)$, one has $y_{k+2}\leq f(x_{k+2})<x_{k+2}$. Hence, to show that $(x_{k+2},y_{k+2})\in P_1$
we only need to prove that $x_{k+2}>0.14$, since $(x_{k+1},y_{k+1})\in P_6$ yields $x_{k+2}<\alpha$.

Observe that~\eqref{eq11} implies
$$x_k\geq\alpha-\frac{\sqrt{\gamma}}{2}=:\Delta;$$
whence (see Figure~\ref{pic}),
\begin{figure}[ht!]
  \begin{center}
    \includegraphics[width=7cm]{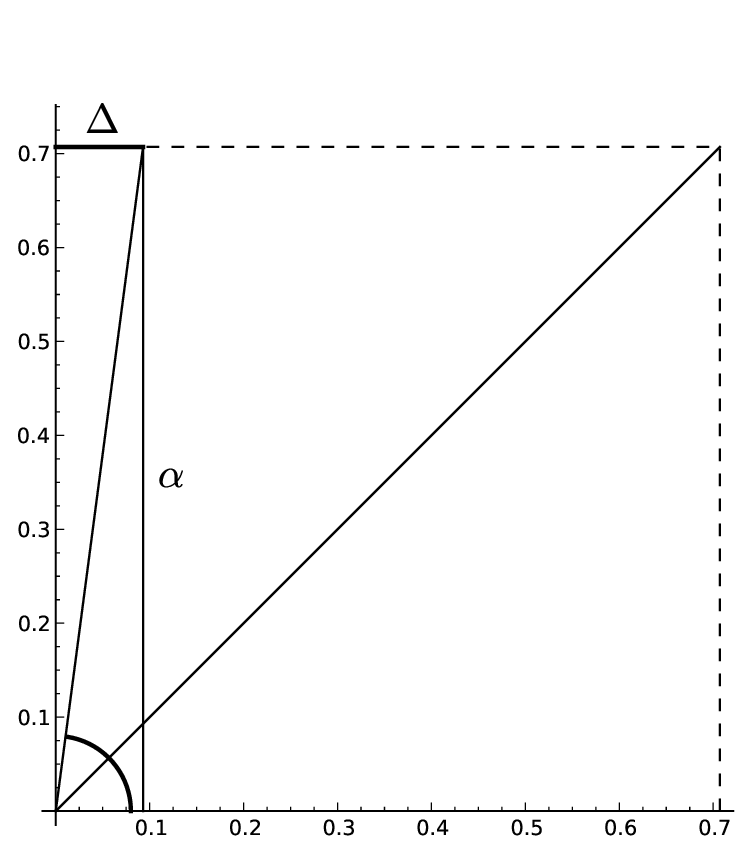}
  \end{center}\caption{The point $(x_k,y_k)$ belongs to the region $P_6$, with $x_k\geq\Delta$.}
  \label{pic}
\end{figure}
$$x_{k+1}=\cos \theta_k\geq\frac{\Delta}{\sqrt{\alpha^2+\Delta^2}}=:\Upsilon,$$
and similarly, since $(x_{k+1},y_{k+1})\in P_6$, one has
$$x_{k+2}=\cos \theta_{k+1}\geq\frac{\Upsilon}{\sqrt{\alpha^2+\Upsilon^2}}\approx 0.18124764381>0.14,$$
and hence, $(x_{k+2},y_{k+2})\in P_1$. Finally, by Lemma~\ref{P56}, one has
$$|(x_{k+2},y_{k+2})-(\alpha,\alpha)|^2\leq \gamma |(x_{k+1},y_{k+1})-(\alpha,\alpha)|^2\leq \frac{\gamma^3}{4}|(x_0,y_0)-(\alpha,\alpha)|^2,$$
and the proof is complete.
\end{proof}

The final two preliminary results tell us that the iterations hit $P_1$ after a finite number of steps whenever the initial point lies in $P_2$ or $P_3$ (with $x_0\leq 1$). Hence,
convergence  on starting within these regions will be a consequence of convergence within $P_1$.

\begin{corollary}\label{Cor_P3}
If $(x_0,y_0)\in P_3$ with $x_0\leq 1$, then there is some  integer $k\in\N$ such that $(x_k,y_k)\in P_1$.
\end{corollary}

\begin{proof}
Notice first that
$$|(x_0,y_0)-(\alpha,\alpha)|^2\leq 2(1-\alpha)^2.$$
By Lemma~\ref{lemma1} we have $(x_1,y_1)\in P_4$, and moreover,
$$|(x_1,y_1)-(\alpha,\alpha)|^2\leq \frac{1}{2}|(x_0,y_0)-(\alpha,\alpha)|^2.$$
By a similar argumentation to the one in the proof of Theorem~\ref{globalP1}, we can find some $k\in\N$ such that $(x_k,y_k)\in P_6$, having also
$$|(x_k,y_k)-(\alpha,\alpha)|^2\leq \frac{\gamma}{2}|(x_0,y_0)-(\alpha,\alpha)|^2;$$
whence,
$$|(x_k,y_k)-(\alpha,\alpha)|^2\leq\gamma(1-\alpha)^2<\frac{\gamma}{4}.$$
Therefore, we have that~\eqref{eq11} holds, and hence,
the rest of reasoning in the proof of Theorem~\ref{globalP1} remains valid, leading to either $(x_{k+1},y_{k+1})\in P_1$ or $(x_{k+2},y_{k+2})\in P_1$.
\end{proof}

\begin{corollary}\label{Cor_P2}
If $(x_0,y_0)\in P_2$ with $x_0\leq 1$, then there is some integer $k\in\N$ such that $(x_k,y_k)\in P_1$.
\end{corollary}

\begin{proof}
From Lemma~\ref{lemma1}(ii) we have $(x_1,y_1)\in P_3\cup P_4$. If $(x_1,y_1)\in P_3$, since $x_1=\cos\theta_0< 1$, we can apply Corollary~\ref{Cor_P3} and we are done. Otherwise
we have $x_1<y_1$. By~\eqref{reverse_x} and~\eqref{reverse_y}, one has
$$1\geq x_0=\frac{x_1(y_1-\alpha)}{\sqrt{1-x_1^2}}+x_1,$$
or, equivalently,
$$1-x_1\geq\frac{x_1(y_1-\alpha)}{\sqrt{1-x_1^2}}> 0.$$
This implies that
$$x_1^2(y_1-\alpha)^2\leq(1-x_1)^2(1-x_1^2).$$
That is,
$$\alpha-\left(\frac{1}{x_1}-1\right)\sqrt{1-x_1^2}\leq y_1\leq\alpha+\left(\frac{1}{x_1}-1\right)\sqrt{1-x_1^2}.$$
Moreover, since $y_0\leq\alpha$, from~\eqref{reverse_y} we get
$$\alpha\geq y_0=y_1-\alpha+\sqrt{1-x_1^2},$$
and hence,
$$y_1\leq 2\alpha-\sqrt{1-x_1^2}.$$
Consider the functions $f(x):=\alpha+\left(1/x-1\right)\sqrt{1-x^2}$ and $g(x):=2\alpha-\sqrt{1-x^2}$ for $x\in(0,1)$. One has $f(x)=g(x)$ if and only if $x$ verifies the equation
$$\frac{\sqrt{1-x^2}}{x}=\alpha,$$
whose unique solution in $(0,1)$ is $\hat x=\sqrt{2/3}$, with $f(\hat x)=g(\hat x)=\sqrt{2}-1/\sqrt{3}=:\hat y$. Observe that $f$ is strictly decreasing in $(0,1)$, while $g$ is
strictly decreasing, since
$$f'(x)=\frac{x^3-1}{x^2\sqrt{1-x^2}}<0\quad\text{and}\quad g'(x)=\frac{x}{\sqrt{1-x^{2}}}>0.$$
Putting all these facts together, we deduce that $y_1\leq\hat y$. Finally, noticing that $g$ is convex and remembering that $x_1<y_1$ with $(x_1,y_1)\in P_4$, one must have (see
Figure~\ref{pic1})
\begin{figure}[htb!]
\begin{center}
\includegraphics[height=7cm]{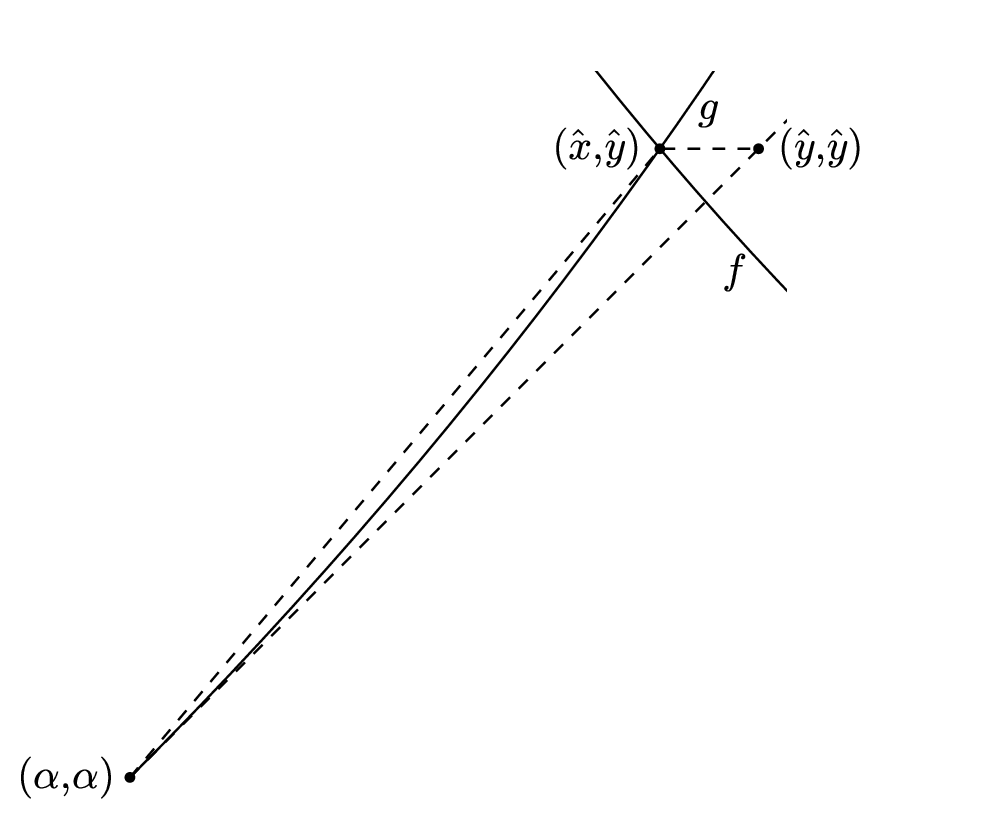}
\end{center}
\caption{The point $(x_1,y_1)$ lies inside the dotted triangle.}
\label{pic1}
\end{figure}
$$|(x_1,y_1)-(\alpha,\alpha)|^2\leq |(\hat y,\hat y)-(\alpha,\alpha)|^2=1-\sqrt{\frac{2}{3}}<\frac{\gamma}{4}.$$
This later inequality, together with Lemma~\ref{P4}, implies that~\eqref{eq11} holds for some $k\in\N$. By a similar argumentation to the the proof of Theorem~\ref{globalP1}, one
has either $(x_{k+1},y_{k+1})\in P_1$ or $(x_{k+2},y_{k+2})\in P_1$.
\end{proof}

We are finally ready to prove our main result.

\bigskip

\begin{proof}[Proof of Theorem~\ref{th:main}]
  Thanks to Lemma~\ref{P4}, Lemma~\ref{P56}, Corollary~\ref{Cor_P3} and Corollary~\ref{Cor_P2}, we may assume without loss of generality that $(x_0,y_0)\in P_1$ (observe that in the region $P_4$
  the distance to the point $(\alpha,\alpha)$ does not increase, while in the region $P_5$ the $x$-coordinate increases); see the left picture in Figure~\ref{fig:convergence}.
  Then Theorem~\ref{globalP1} applies, and there is
  some $m\in\N$ such that $(x_m,y_m)\in P_1$. Moreover, inequality~\eqref{eq7} holds for all $k\leq m$. Since $\gamma^3/4<1$, the sequence generated by the Douglas-Rachford
  iteration~\eqref{DR_2} must converge to the point $(\alpha,\alpha)$.
\end{proof}

\begin{remark}
In fact we have shown the convergence of the algorithm in a bigger region: if either $(x_0,y_0)\in P_1$, or $(x_0,y_0)\in P_4$ with $(x_0-\alpha)^2+(y_0-\alpha)^2\leq
(\varepsilon-\alpha)^2+(\sqrt{1-\varepsilon^2}-\alpha)^2$, for $\varepsilon:=\left(1-2^{-1/3}\right)^{3/2}$, see the picture in the left in Figure~\ref{fig:convergence}.
\end{remark}

\begin{figure}[htb!]
  \begin{center}
    \includegraphics[width=6cm]{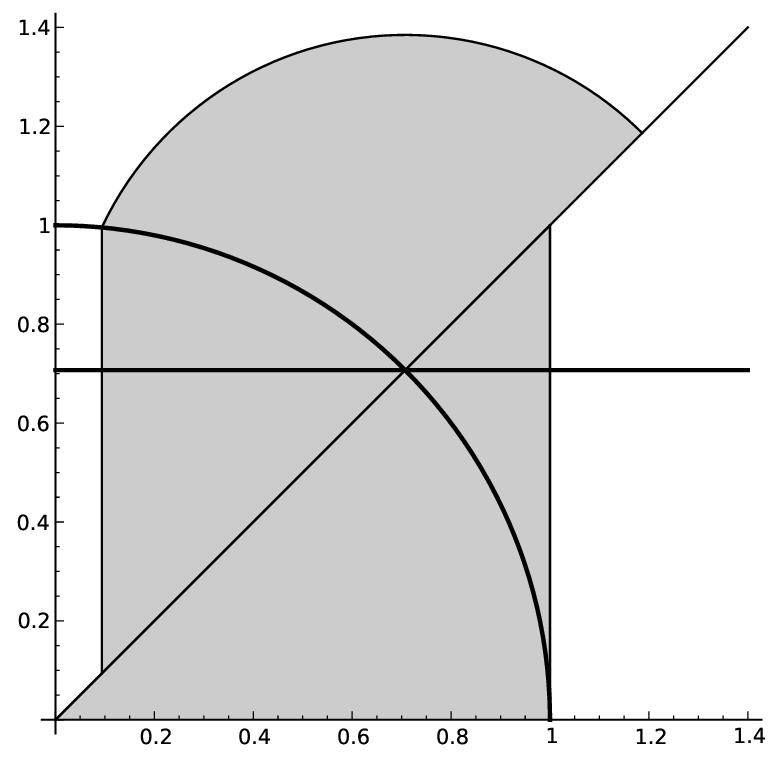}
    \includegraphics[width=6cm]{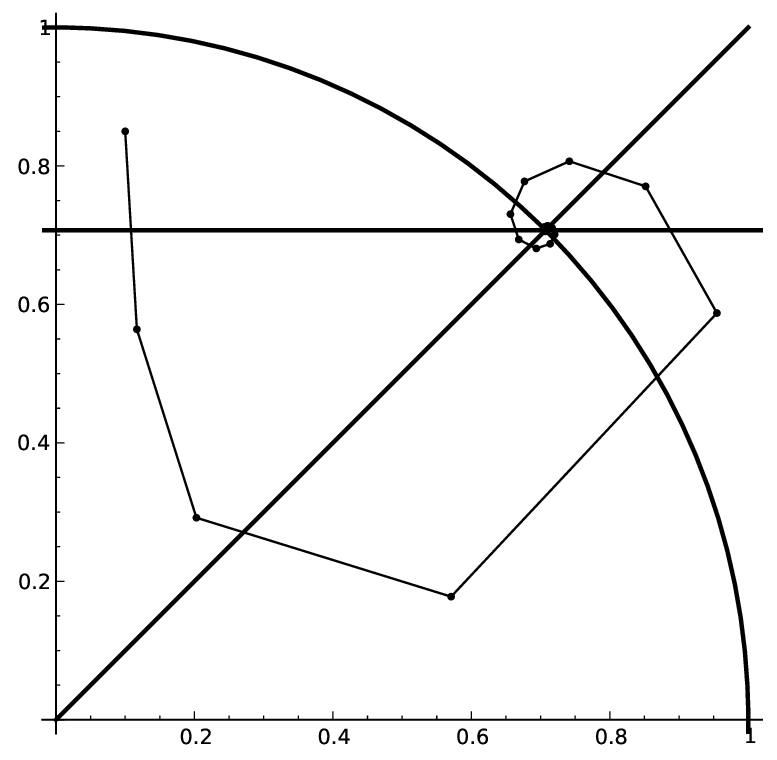}
  \end{center}
  \caption{The picture in the left shows the regions of convergence in Theorem~\ref{th:main} for the Douglas-Rachford algorithm. The picture in the right illustrates
  an example of a convergent sequence generated by the algorithm.}
  \label{fig:convergence}
\end{figure}

Finally, recalling Proposition~\ref{prop2}, Theorem~\ref{globalP1}, Corollary~\ref{Cor_P3} and Corollary~\ref{Cor_P2}, 
we can exhibit a more general dichotomy.

\begin{corollary}\label{Cor_F}
Given $(x_0,y_0)\in\R^2$ with $x_0>0$, consider the sequence generated by the Douglas-Rachford scheme~\eqref{DR_2}. Then, either the iteration visits $P_0$  infinitely often or it converges to the point $(\alpha,\alpha)$.
\end{corollary}

Of course we believe the first case never occurs.

\section*{Acknowledgements}
We would like to thank Juan Carlos Ram\'irez for pointing out a typo in the statement of Theorem~2. This is the only change with respect to the previous version.

\end{document}